\documentclass[11pt,reqno]{amsart}
\usepackage{amsmath,amssymb}
\usepackage{graphicx,psfrag}
\usepackage{amsthm}
\usepackage{enumerate,placeins,wrapfig,booktabs,chngcntr}
\usepackage{subcaption}
\usepackage{dsfont}
\usepackage{appendix}

\usepackage[colorlinks=true, linkcolor=blue, citecolor=green]{hyperref}
\newtheorem{thm}{Theorem}[section] 
\newtheorem{conj}[thm]{Conjecture}

\newtheorem{cor}[thm]{Corollary}
\newtheorem{lem}[thm]{Lemma}
\newtheorem{prop}[thm]{Proposition}

\newtheorem{defn}[thm]{Definition}

  \allowdisplaybreaks[4]

\begin{document}

    \title[Transitive partially hyperbolic diffeomorphisms in dimension three]{Transitive partially hyperbolic diffeomorphisms in dimension three}

	\author{Ziqiang Feng}
    \address{Beijing International Center for Mathematical Research, Peking University, Beijing, 100871, China.}
    \email{zqfeng@pku.edu.cn}

    \thanks{This work was partially supported by National Key R\&D Program of China 2022YFA1005801.}

	\begin{abstract}
        We prove that any $C^{1+\alpha}$ transitive conservative partially hyperbolic diffeomorphism of a closed 3-manifold with virtually solvable fundamental group is ergodic. Consequently, in light of \cite{FP-classify}, this establishes the equivalence between transitivity and ergodicity for $C^{1+\alpha}$ conservative partially hyperbolic diffeomorphisms in \emph{any} closed 3-manifold. Moreover, we provide a characterization of compact accessibility classes under transitivity, thereby giving a precise classification of all accessibility classes for transitive 3-dimensional partially hyperbolic diffeomorphisms.
	\end{abstract}

    \date{\today}

	\maketitle


	\vspace{.5cm}



\tableofcontents

\section{Introduction}

The study of partially hyperbolic dynamics has attracted tremendous interest as one of the most significant extensions of uniformly hyperbolic systems.  
A partially hyperbolic diffeomorphism $f:M\to M$ retains uniformly hyperbolic behavior in its extreme bundles, $E^s$ and $E^u$, within its invariant splitting $TM=E^s\oplus E^c\oplus E^u$, but exhibits dominated pointwise spectrum in the middle bundle $E^c$. The presence of the center bundle $E^c$ and the dynamical flexibility of the center action along $E^c$ enlarge the difference between partially hyperbolic systems and uniformly hyperbolic ones.

While uniformly hyperbolic systems have been well understood in many aspects, certain fundamental properties remain far from being known in the partially hyperbolic context. The purpose of this paper is to reveal a intrinsic relation between the transitivity and ergodicity for partially hyperbolic diffeomorphisms in dimension three. Besides, we are going to provide a classification of accessibility classes.

\subsection{Ergodicity}

Transitivity and ergodicity are two fundamental properties from the topological and measure-theoretical perspectives, respectively, for general dynamical systems, not only within partially hyperbolic dynamics. Although partially hyperbolic systems lack structural stability, they arise naturally in the study of stable ergodicity \cite{GPS94} and robust transitivity \cite{DPU}.

The verification of ergodicity for uniformly hyperbolic systems is built upon Hopf's argument \cite{Hopf}, extended by Anosov and Sinai \cite{Anosov67,AnosovSinai67}. Adapting this argument to the partially hyperbolic context becomes highly non-trivial due to the lack of hyperbolicity in the center direction, which complicates the ergodic problem for such systems. Based on the work of \cite{GPS94}, Pugh and Shub proposed a way to examine ergodicity through accessibility-a property that captures pathwise hyperbolicity-and conjectured that ergodic partially hyperbolic diffeomorphisms constitute an open and dense set among conservative ones. Significant progress on the Pugh--Shub conjecture, including results from \cite{HHU08invent,BW10,ACW}, has established the prevalence of ergodicity in dimension three.

In 2024, Meysam Nassiri presented us with a conjecture that he had long held, which connects the two fundamental properties, transitivity and ergodicity, in the partially hyperbolic context. The conjecture was initially proposed in the three-dimensional case but is expected to hold for systems with a one-dimensional center bundle. We formulate this more general conjecture below.

\begin{conj}[Nassiri]\label{conj}
    Every $C^{1+\alpha}$ transitive conservative partially hyperbolic diffeomorphism with one-dimensional center on a closed manifold is ergodic.
\end{conj}

This conjecture was previously established in \cite{Feng25} under the quasi-isometric center condition. This paper provides a complete affirmative answer to the conjecture stated above in the three-dimensional setting.

\begin{thm}\label{transitive<=>ergodic}
    Let $f: M\to M$ be a $C^{1+\alpha}$ conservative partially hyperbolic diffeomorphism of a closed 3-manifold. Then, $f$ is transitive if and only if it is ergodic.
\end{thm}

To the best of our knowledge, the conjecture remains widely open in the general one-dimensional center case.

It is worth noting that this theorem is sharp within the given context, as it no longer holds if one relaxes the transitivity condition to having the entire manifold as a single non-wandering set. Indeed, the non-wandering property automatically holds under the conservative condition. Numerous conservative (and thus non-wandering) yet non-ergodic partially hyperbolic systems exist; for instance, the product of an Anosov diffeomorphism with the identity map on the 3-torus and the time-one map of a suspension Anosov flow provide two standard examples.

The work of \cite{FP-classify} and \cite{FP-ajm} obtains accessibility, and thus ergodicity using \cite{HHU08invent,BW10}, for conservative partially hyperbolic diffeomorphisms in closed 3-manifolds with non-virtually solvable fundamental groups. This leads our work to be sufficient to prove ergodicity in the solvable case.

\begin{thm}\label{sol}
    Let $f: M\to M$ be a $C^{1+\alpha}$ conservative partially hyperbolic diffeomorphism of a closed 3-manifold with virtually solvable fundamental group. If $f$ is transitive, then it is ergodic.
\end{thm}

As a consequence of this theorem, the $C^2$-regularity condition in \cite[Theorem 1.3]{FU-noperiodic} can be relaxed to be $C^{1+\alpha}$, which shows the ergodicity for all conservative three-dimensional partially hyperbolic diffeomorphisms without periodic points.

We note that all conservative partially hyperbolic diffeomorphisms in the solvable case have been classified \cite{HP-sol} into three categories: DA diffeomorphisms, skew products, and discretized Anosov flows. However, ergodicity does not follow immediately from this classification. In addition to the examples mentioned above (the product of an Anosov diffeomorphism with the identity and the time-one map of a suspension flow), one can manually insert accessibility classes to break ergodicity. For a simple example of such a construction, consider firstly the product of an Anosov diffeomorphism with the identity. Then, by blowing up the origin of the circle along the 2-torus and inserting an open accessibility class, one obtains a non-ergodic diffeomorphism with both compact and non-compact accessibility classes. Similarly, one can destroy ergodicity by inserting open accessibility classes into the product of an Anosov diffeomorphism with an irrational rotation, thereby turning an ergodic example into a non-ergodic one. Note that the skew product structure is preserved in these constructions, which implies the coexistence of ergodic and non-ergodic skew product diffeomorphisms. Analogous constructions can be carried out for discretized suspension Anosov flows. Theorem~\ref{sol} shows that transitivity rules out such constructions. For a precise description of their accessibility classes, see Theorem~\ref{c-dim-one}.

A related problem to Conjecture~\ref{conj} is the Hertz–Hertz–Ures Ergodicity Conjecture, which states that all conservative three-dimensional partially hyperbolic diffeomorphisms without an $su$-torus are ergodic.  
We refer the reader to \cite{HHU08nil,HammerlindlUres,GS-DA,HamHU20,FP-adv,FP-ajm,FU-noperiodic,Feng25,FU-id} for previous developments on this conjecture. We emphasize that this conjecture has recently been completely resolved by \cite{FP-classify}.

\subsection{Classification of accessibility classes}

As suggested by the Pugh–Shub conjecture, accessibility serves as a powerful tool for establishing ergodicity. This has been verified for partially hyperbolic systems with a one-dimensional center \cite{HHU08invent} and under center-bunching conditions \cite{BW10}. From a purely topological perspective, it provides a way to strengthen dynamical recurrence from non-wandering to transitivity \cite{Brin75}. Moreover, accessibility enables the generalization of periodic obstructions to solving the cohomological equation and Livšic regularity problems in the partially hyperbolic context \cite{Wilkinson13}. Furthermore, its promising potential has been demonstrated in deriving significant properties of special ergodic measures, as detailed in \cite{HHTU12, VY13, AVW-flow, CP-invariance}, and in understanding centralizer groups \cite{DWX-centralizer, GSXZ22}. We refer the reader to \cite{Wilkinson10, CHHU18} for further discussion.

Regardless of volume preservation, accessibility has been proven to be a prevalent property among partially hyperbolic systems \cite{DolgopyatWilkinson03, BHHTU08, AV20}. This makes the study of the meager set of non-accessible systems more involved in the problem of determining which partially hyperbolic diffeomorphisms are accessible. A naturally arising question is to completely characterize all accessibility classes for non-accessible partially hyperbolic diffeomorphisms.

We provide a precise description on the set of compact accessibility classes as follows:

\begin{thm}\label{c-dim-one}
    Let $f: M\to M$ be a $C^{1}$ partially hyperbolic diffeomorphism of a closed manifold with $\dim E^c=1$. Assume that $f$ admits a compact accessibility class. If $f$ is transitive, then either
    \begin{enumerate}
        \item there exist finitely many compact accessibility classes, all of which are $f$-periodic; or
        \item there exists an $su$-foliation by compact leaves.
    \end{enumerate}
\end{thm}

Combining results in \cite{FP-ajm,FP-classify} (see Theorem~\ref{CAF}), we have the following consequence, giving a classification of accessibility classes in dimension three. See Section~\ref{subsec-AB} for the definition of infra-$AB$-systems.

\begin{cor}\label{transtive}
    Let $f: M\to M$ be a $C^{1}$ transitive partially hyperbolic diffeomorphism of a closed 3-manifold. Then, exactly one of the following cases occurs:
    \begin{enumerate}
        \item $f$ is accessible;
        \item there exist finitely many $su$-tori, all of which are $f$-periodic;
        \item there exists an $su$-foliation by tori;
        \item $f$ is a DA diffeomorphism with minimal $su$-foliation on $\mathbb{T}^3$. Moreover, if $f$ is $C^{1+\alpha}$, then it is an Anosov diffeomorphism.
    \end{enumerate}
    Furthermore, $f$ is an infra-$AB$-system in cases (2) and (3).
\end{cor}

Under the condition $NW(f)=M$, \cite[Theorem 1.6]{HHU08nil} provides a classification of accessibility classes, including the case of an invariant lamination by non-compact accessibility classes. Theorem~\ref{transtive} implies that this lamination case cannot occur (see also Theorem~\ref{NW}). In the absence of periodic points, we have shown in \cite{FU-noperiodic2} that accessibility and the existence of an $su$-torus form a dichotomy without requiring $NW(f)=M$. Moreover, we provided a complete description of accessibility classes for partially hyperbolic diffeomorphisms with one-dimensional center in the absence of periodic points \cite{FU-noperiodic2}. 

We refer the reader to \cite[Corollary 1.5]{HS-DA} for the case with solvable fundamental group. Besides Theorem~\ref{c-dim-one}, the novelty of Corollary~\ref{transtive} also relies on the classification in \cite{HP-sol,FP-classify}, which requires $NW(f)=M$. Note that, following results from \cite{GS-DA,HS-DA}, the reduction from DA diffeomorphisms to Anosov diffeomorphisms in the last case of Corollary~\ref{transtive} relies on the $C^{1+\alpha}$ condition. It would be interesting to investigate whether this condition can be relaxed, which essentially requires checking the accessibility for non-Anosov DA diffeomorphisms.

\subsection{Organization}

The structure of this paper is as follows. In Section~\ref{section-pre}, we provide foundational concepts and preliminary results that will be used in subsequent sections. The proofs of Theorem~\ref{c-dim-one}, and Corollary~\ref{transtive} are included in Section~\ref{section-acc}. Section~\ref{section-erg} focuses on the equivalence between transitivity and ergodicity, providing the proof of Theorem~\ref{sol} and Theorem~\ref{transitive<=>ergodic}. In the Appendix, we include additional equivalence relations regarding chain transitivity.

\section{Preliminaries}\label{section-pre}

In this section, we introduce preliminary concepts and relevant results that will be used throughout the rest of the paper.  

\subsection{Partial hyperbolicity}

Let $M$ be a compact connected Riemannian manifold. We say a diffeomorphism $f: M\to M$ is \emph{(strongly pointwise) partially hyperbolic} if the tangent bundle of $M$ splits into three nontrivial invariant subbundles $TM=E^s \oplus E^c \oplus E^u$ such that for an adapted metric and for each $x\in M$ and each unit vectors $v^\sigma \in E^\sigma_x$ ($\sigma= s, c, u$), 
\begin{equation*}
	\|Df(x)v^s\|<1<\|Df(x)v^u\| \\
	\quad and \quad
	\|Df(x)v^s\|< \|Df(x)v^c\|< \|Df(x)v^u\|.
\end{equation*}

The existence of stable and unstable foliations plays a fundamental role in the study of Anosov systems. The unique integrability of the strong bundles for partially hyperbolic diffeomorphisms was established in the foundational work of \cite{BrinPesin74, PughShub72}. We denote by $\mathcal{F}^s$ the stable foliation of $f$ and refer to $\mathcal{F}^s(x)$ as the stable leaf through the point $x\in M$. Similarly, we use $\mathcal{F}^u$ and $\mathcal{F}^u(x)$ for the unstable foliation and leaves. A set is \emph{$s$-saturated} (resp. \emph{$u$-saturated}) if it is a union of stable (resp. unstable) leaves, and \emph{$su$-saturated} if it satisfies both saturation.  
The \emph{accessibility class} $AC(x)$ of a point $x\in M$ is the minimal $su$-saturated set containing $x$. Any pair of points lying in the same accessibility class can be connected by an \emph{$su$-path}, which consists of connected curves that are piecewise tangent to $E^s\cup E^u$.

A partially hyperbolic diffeomorphism $f$ is \emph{accessible} if $AC(x)=M$ for every $x\in M$; in other words, any two points can be joined by an $su$-path. Let $\Gamma(f)$ denote the set of non-open accessibility classes.  
Then $f$ is accessible if and only if $\Gamma(f)=\emptyset$.

\begin{thm}\cite{HHU08invent}\label{Gamma-f}
    Let $f:M\to M$ be a non-accessible partially hyperbolic diffeomorphism with one-dimensional center. Then $\Gamma(f)$ forms a codimension-one, compact, $f$-invariant lamination.
\end{thm}

The following recent result provides a classification of partially hyperbolic diffeomorphisms and their accessibility.
\begin{thm}\cite{FP-ajm,FP-classify}\label{CAF}
    Let $f: M\to M$ be a $C^1$ partially hyperbolic diffeomorphism of a closed 3-manifold whose fundamental group $\pi_1(M)$ is not virtually solvable. If $f$ is (chain) transitive, then $f$ is a strong collapsed Anosov flow. If $NW(f)=M$, then $f$ is accessible.
\end{thm}

Combining this with Brin's result \cite{Brin75}, we obtain the following corollary:
\begin{cor}
    Let $f: M\to M$ be a $C^1$ partially hyperbolic diffeomorphism of a closed 3-manifold whose fundamental group $\pi_1(M)$ is not virtually solvable. Then, $NW(f)=M$ if and only if $f$ is transitive.
\end{cor}

It would be interesting to investigate whether the condition $NW(f)=M$ in the corollary above can be relaxed to chain-recurrence, although a counterexample exists on the 3-torus \cite{GanShi-chaintransitive}.

\subsection{$AB$-systems}\label{subsec-AB}

Let us recall a typical class of partially hyperbolic diffeomorphisms with one-dimensional center, known as $AB$-systems, which were originally introduced in \cite{Hammerlindl17}. This class shares certain features with skew products but is more general in several respects. 

Skew products constitute a typical class of diffeomorphisms that fiber over a lower-dimensional system. We say that a diffeomorphism $f:M\to M$ of a compact manifold is a \emph{skew product over a map $A:N\to N$} if it preserves the fibration $\pi:M\to N$ and the projection map $A$ satisfies $\pi\circ f=A\circ\pi$, where $N$ is a topological manifold of dimension lower than that of $M$. Note that the fibration of a skew product is not necessarily a trivial bundle or even a principal bundle. In the partially hyperbolic context, one usually considers skew products whose base maps are hyperbolic. 

Given two automorphisms $A$ and $B$ acting on a compact nilmanifold $N$ such that $A$ is hyperbolic and $AB=BA$, an \emph{$AB$-prototype} is defined as the diffeomorphism  
$$f_{AB}: M_B\to M_B, \quad (v,t)\mapsto (Av, t)$$  
on the manifold  
$$M_B=N\times \mathbb{R}\,/\,(v,t)\sim(Bv,t-1).$$  
It is clear that $f_{AB}$ is a partially hyperbolic diffeomorphism, and the fibration $\pi: M_B\to N$ corresponds to the integral foliation of its one-dimensional center bundle.

\begin{defn}
    A partially hyperbolic diffeomorphism $f:M\to M$ is an \emph{AB-system} if 
    \begin{enumerate}
        \item it preserves the orientation of the center bundle $E^c$;
        \item $f$ is leaf conjugate to an $AB$-prototype $f_{AB}$ by a homeomphism $h$; and
        \item $h\circ f\circ h^{-1}$ is homotopic to $f_{AB}$.
    \end{enumerate}
\end{defn}

A broad class of examples consists of skew products over an Anosov diffeomorphism on a nilmanifold. Under the orientation-preserving condition, $AB$-systems can be viewed as a generalization of skew products with a trivial one-dimensional fiber bundle. However, $AB$-systems also encompass other classes of partially hyperbolic systems; for instance, the case where $B$ itself is hyperbolic. 

The $AB$-prototype $f_{AB}$ serves as a linear model for $AB$-systems, as they share the same conjugacy and homotopy classes up to center leaves. For a given $f_{AB}$, all $AB$-systems may exhibit diverse dynamical behaviors, particularly in the center direction. It is worth noting that the ambient manifold $M$ of an $AB$-system is only guaranteed to be homeomorphic to $M_B$, not necessarily diffeomorphic \cite{FarrellJones,FarrellGogolev}. In general, the center bundle $E^c$ is not necessarily orientable; and even if it is, the partially hyperbolic diffeomorphism may not preserve a given orientation. Moreover, condition (3) in the definition above is not redundant—see the \cite[Erratum]{Hammerlindl17} for examples that are not homotopic. To accommodate these subtleties as well as skew products on infra-nilmanifolds, we may consider the following more general class.

A diffeomorphism $f$ is called an \emph{infra-$AB$-system} if some iterate of $f$ lifts to an $AB$-system on a finite cover. The following lemma implies that a partially hyperbolic diffeomorphism satisfying condition (2) is, in fact, an infra-$AB$-system.

\begin{lem}
    If $f$ is leaf conjugate to an $AB$-prototype $f_{AB}$, then an iterate $\hat{f}^k$ of a lift to a finite cover is an $AB$-system (i.e., $f$ is an infra-$AB$-system).
\end{lem}
\begin{proof}
    Since the center bundle $E^c$ is one-dimensional, we may pass to a double cover of $M$ to make $E^c$ orientable. After fixing an orientation, there exists a finite iterate $\hat{f}^n$ of a lift to this double cover that preserves the orientation of $E^c$. Then $\hat{f}^n$ is a partially hyperbolic diffeomorphism satisfying conditions (1) and (2). By \cite[Proposition 2 of Erratum]{Hammerlindl17}, one can further lift $\hat{f}^n$ to a finite cover so that all three conditions are satisfied. Hence, $f$ is an infra-$AB$-system.
\end{proof}

The accessibility classes of $AB$-systems are well understood, as described in the following result:
\begin{thm}\cite[Theorem 2.5]{Hammerlindl17}\label{Ham-2.5}
    Every $AB$-system is either accessible or possesses a compact accessibility class.
\end{thm}

\section{Accessibility classes}\label{section-acc}

This section is devoted to the characterization of accessibility classes for partially hyperbolic diffeomorphisms with one-dimensional center. We include proofs of Theorem~\ref{c-dim-one} and Corollary~\ref{transtive}.

The following result provides a description of accessibility classes and non-accessible partially hyperbolic diffeomorphisms in dimension three.

\begin{thm}\label{NW}
    Let $f: M\to M$ be a $C^{1}$ partially hyperbolic diffeomorphism of a closed 3-manifold with $NW(f)=M$. Then, we have a trichotomy:
    \begin{enumerate}
        \item $f$ is accessible;
        \item $f$ is an infra-$AB$-system admitting an $su$-torus;
        \item $f$ is a DA diffeomorphism with minimal $su$-foliation on $\mathbb{T}^3$. Moreover, if $f$ is $C^{1+\alpha}$, then it is an Anosov diffeomorphism.
    \end{enumerate}
\end{thm}

This theorem essentially follows from the classification given by Theorem~\ref{CAF} and \cite{HP-sol}. We sketch its proof for the sake of completeness.

\begin{proof}
    After possibly passing to a finite cover and an iterate, we may assume that $M$ and the bundle $E^c$ are orientable, with orientations preserved by $f$. There is no loss of generality because each possibility listed in the desired trichotomy persists under finite lifts and iterates. When the fundamental group $\pi_1(M)$ is not virtually solvable, the diffeomorphism $f$ must be accessible by Theorem~\ref{CAF}.

    The result in \cite{HHU08nil} establishes that $f$ is accessible provided that $\pi_1(M)$ is virtually nilpotent but $M$ is not the 3-torus. If $\pi_1(M)$ is virtually solvable but not virtually nilpotent, then $f$ has an iterate that is a discretized Anosov flow \cite{HP-sol} and $M$ is finitely covered by a torus bundle over the circle \cite{Parwani10}. In this case, after taking a finite cover and an iterate if necessary, $f$ is an $AB$-system, which by Theorem~\ref{Ham-2.5} is either accessible or has a compact submanifold tangent to $E^s\oplus E^u$. Any such compact submanifold in our setting must be a 2-torus, as it admits a one-dimensional unstable foliation. This yields the first two possibilities in the trichotomy when $M$ is not the 3-torus.

    When the manifold $M$ is $\mathbb{T}^3$ up to a finite lift, the diffeomorphism $f$ induces an action $f_*:H_1(M,\mathbb{Z})\to H_1(M, \mathbb{Z})$ on the first homology group $H_1(M,\mathbb{Z})\cong \mathbb{Z}^3$. There is a unique linear automorphism $L:\mathbb{T}^3\to \mathbb{T}^3$ inducing the same action as $f_*$, called the \emph{linearization} of $f$. The map $L$ is partially hyperbolic \cite{BI09} with three distinct eigenvalues $|\lambda_1|<|\lambda_2|<|\lambda_3|$, where $|\lambda_1|<1<|\lambda_3|$. If $|\lambda_2|=1$, then $f$ is leaf conjugate to a skew product over an Anosov automorphism on $\mathbb{T}^2$ and hence is an $AB$-system. Then, Theorem~\ref{Ham-2.5} implies that $f$ is either accessible or has a compact accessibility class. As noted above, such a compact accessibility class is an $su$-torus. If $|\lambda_2|\neq 1$, then the linear map $L$ is an Anosov automorphism sharing the same homotopy class as $f$. Thus, $f$ is a DA diffeomorphism. Moreover, it admits a minimal $su$-foliation by \cite[Theorem 1.6]{HHU08nil} since $su$-torus cannot exist. By \cite{HS-DA}, a non-accessible $C^{1+\alpha}$ DA diffeomorphism must be Anosov.

    This completes the proof.
\end{proof}

Before proceeding with the proof of Theorem~\ref{c-dim-one}, we first assume it to establish Corollary~\ref{transtive}.

\begin{proof}[Proof of Corollary~\ref{transtive}]
    Transitivity implies the condition $NW(f)=M$ required in Theorem~\ref{NW}. Then we directly apply Theorem~\ref{NW} to conclude that $f$ is either accessible or a DA diffeomorphism on $\mathbb{T}^3$, provided that there is no $su$-torus. If an $su$-torus exists, it is a compact accessibility class. By Theorem~\ref{c-dim-one}, either there are only finitely many compact accessibility classes, all of which are $f$-periodic, or $f$ preserves an $su$-foliation by compact leaves. Any such compact accessibility class must be a 2-torus, as it is foliated by one-dimensional unstable manifolds. This completes the proof.
\end{proof}

For a general lamination, the collection of its compact leaves is not necessarily a compact set, whereas this property holds when its codimension is one.

\begin{thm}\label{Hae62}\cite{Hae62}  
    The set of compact leaves of a $C^0$ codimension-one lamination forms a compact sublamination.  
\end{thm}  

We remark that this property was originally stated for foliations and, in fact, it also holds for laminations; see for instance \cite{HectorHirsch}.

Now, let us present our proof of Theorem~\ref{c-dim-one}.

\begin{proof}[Proof of Theorem~\ref{c-dim-one}]
    Denote by $\Gamma(f)$ the set of non-open accessibility classes. By Theorem~\ref{Gamma-f}, $\Gamma(f)$ is a codimension-one compact $f$-invariant lamination. Let $\Lambda$ be the set of all compact accessibility classes. Then, by Theorem~\ref{Hae62}, $\Lambda$ is a compact $f$-invariant subset of $\Gamma(f)$. If $\Lambda$ consists of finitely many accessibility classes, we can find a common period $N\in \mathbb{N}$ such that $f^N$ fixes every class in $\Lambda$.

    Suppose that $\Lambda$ contains infinitely many compact accessibility classes and $\Lambda\neq M$. Let $R\subset \Lambda^c$ be a maximal connected complementary region of $\Lambda$. Its boundary consists of finitely many compact accessibility classes (see for instance \cite[Chapter V, Corollary 3.2.4]{HectorHirsch}), which we denote by $T_i$, $i=1,\dots,n$. Note that $R$ has non-empty interior because $\Lambda$ is compact. Since every point of $M$ is non-wandering, for any $x\in \mathrm{int}(R)$ there exists an integer $k\in \mathbb{N}$ such that $f^k(x)\in \mathrm{int}(R)$. As both $\Lambda$ and its complement $\Lambda^c$ are $f$-invariant, it follows that the region $R$ is $f^k$-invariant. Moreover, the closure $\hat{R}:=R\cup \left(\bigcup\limits_{i=1}^{n}T_i\right)$ is also $f^k$-invariant. 
    
    Define $V:=\bigcup\limits_{j=1}^kf^j(\hat{R})$, which is clearly a closed $f$-invariant set. Transitivity implies that $V=M$. Indeed, if $V$ had a non-empty complement $V^c$, then $V^c$ would be an open set containing a small open ball $B\subset V^c$. By construction, $V$ has non-empty interior and thus contains a small open ball $D\subset V$. By transitivity, there exists an integer $N\in \mathbb{N}$ such that $f^N(D)\cap B\neq \emptyset$. However, $V$ and $V^c$ are disjoint $f$-invariant sets, a contradiction. Hence, $V=M$. It follows that $\Lambda$ consists of finitely many compact accessibility classes, namely $\{f^j(T_i)\}_{1\leq i\leq n,1\leq j\leq k}$. This contradicts the assumption that $\Lambda$ contains infinitely many classes. Therefore, $\Lambda$ must be a foliation by compact leaves, which completes the proof.
\end{proof}

\section{Ergodicity}\label{section-erg}

This section is devoted to proving Theorem~\ref{transitive<=>ergodic}, which establishes the equivalence of transitivity and ergodicity for conservative partially hyperbolic diffeomorphisms in dimension three. It immediately follows from Proposition~\ref{ergodic=>transitive}, Theorem~\ref{sol}, and Theorem~\ref{CAF}. In particular, we also provide an arbitrarily dimensional result for ergodicity (Theorem~\ref{ergodic-AB}).

The first result is well-known and holds for any continuous map (not necessarily invertible or differentiable) on a connected space (not necessarily a compact manifold).
\begin{prop}\label{ergodic=>transitive}
    Let $f:M\to M$ be a conservative continuous map on a connected manifold (not necessarily compact). If $f$ is ergodic, then it is transitive.
\end{prop}
\begin{proof}
    Let $\mu$ be a smooth probability measure equivalent to the Lebesgue measure on $M$. For any two open sets $U$ and $V$, we have $\mu(U)>0$ and $\mu(V)>0$. By ergodicity, the set $\Lambda:=\bigcup_{i\in \mathbb{Z}}f^i(U)$ has full $\mu$-measure. Hence, we have $\Lambda\cap V\neq \emptyset$. Thus, there exists an integer $i\in \mathbb{Z}$ such that $f^i(U)\cap V\neq\emptyset$, which implies that $f$ is transitive.
\end{proof}

It remains to prove ergodicity under the assumption of transitivity when the fundamental group is virtually solvable (Theorem~\ref{sol}). Before doing so, we recall the following result, which provides a description of all accessibility classes of transitive $AB$-systems.

\begin{thm}\cite[Theorem 2.4]{Hammerlindl17}\label{Ham-2.4}
    Let $f:M\to M$ be an $AB$-system with $NW(f)=M$. Then, one of the following holds.
    \begin{enumerate}
        \item $f$ is accessible and transitive.
        \item $E^s$ and $E^u$ are jointly integrable and $f$ is topologically conjugate to the map $A\times R_{\theta}:M_B\to M_B$, $(v,t)\mapsto (Av, t+\theta)$ for some $\theta\in \mathbb{R}$. Furthermore, $f$ is transitive if and only if $\theta$ is irrational.
        \item There are $n\geq 1$, a continuous surjection $p:M\to\mathbb{S}^1$, and a non-empty open subset $U\subsetneq\mathbb{S}^1$ such that
        \begin{itemize}
            \item for every connected component $I$ of $U$, $p^{-1}(I)$ is a single open accessibility class of $f$ that is $f^n$-invariant and homeomorphic to $N\times I$; and
            \item for every $t\in \mathbb{S}^1\setminus U$, $p^{-1}(t)$ is an $f^n$-invariant submanifold tangent to $E^s\oplus E^u$ and homeomorphic to $N$.
        \end{itemize}
    \end{enumerate}
\end{thm}

We note that the condition of being a single open accessibility class in case (3) was not explicitly stated in \cite[Theorem 2.4]{Hammerlindl17}, but it is implied by the proof. Indeed, case (3) follows from \cite[Theorem 2.6]{Hammerlindl17}, which assumes that $f$ is a non-accessible $AB$-system with at least one periodic compact accessibility class. Suppose that $p^{-1}(I)$ is not a single open accessibility class for some connected component $I\subset U$. There are two possibilities: either there exists an open set $V\subset p^{-1}(I)$ such that $\overline{f^n(V)}\subset V$ and $\bigcap_{k\in \mathbb{Z}}f^{kn}(V)=\emptyset$; or there are uncountably many non-compact accessibility classes in $p^{-1}(I)$ and a constant $\lambda\neq 1$ such that $f^n|_{p^{-1}(I)}$ is semiconjugate to $N\times \mathbb{R}\to N\times \mathbb{R}$, $(v,t)\mapsto (Av, \lambda t)$. In the first case, one deduces that $NW(f^n|_{p^{-1}(I)})\neq p^{-1}(I)$, contradicting the assumption $NW(f)=M$. In the latter case, the condition $\lambda\neq 1$ implies that $f^n|_{p^{-1}(I)}$ is either contracting or expanding in the fiber direction, which is impossible given $NW(f)=M$. Thus, $f^n|_{p^{-1}(I)}$ must be accessible, and hence $p^{-1}(I)$ is a single accessibility class of $f$ for each connected component $I\subset U$.

It is shown in \cite{HHU08invent,BW10} that (essentially) accessible volume-preserving $C^2$ partially hyperbolic diffeomorphisms with one-dimensional center are ergodic. This has been improved by \cite{Brown22} to the $C^{1+\alpha}$ case by establishing the smoothness of stable holonomies inside center-stable manifolds. We state an alternative version of this result below, where ergodicity is restricted to an invariant subset $D\subset M$ instead of the entire manifold $M$.

\begin{thm}\cite[Theorem B]{HHU08invent}\cite[Theorem 0.1]{BW10}\cite[Theorem 1.1]{Brown22}\label{acc->erg}
    Let $f:M\to M$ be a $C^{1+\alpha}$ partially hyperbolic diffeomorphism of a closed manifold with $dimE^c=1$ and $D\subset M$ be an $f$-invariant open accessibility class. Assume that $f$ preserves the $D$-volume measure $\mu_D:=\frac{1}{\mu(D)}\mu\circ\mathds{1}_D$, where $\mu$ is the volume measure on $M$. Then, $f$ is ergodic with respect to $\mu_D$; that is, for any $\phi\in C^0(M)$, we have 
    \begin{equation*}
        \tilde{\phi}(x)=\lim\limits_{n\to\infty}\frac{1}{n}\sum\limits_{i=0}^{n-1}\phi\circ f^i(x)=\int_D \phi\mathrm{d}\mu_D, \quad \mu_D-a.e.x.
    \end{equation*}
\end{thm}

The following result provides ergodicity in any dimension.
\begin{thm}\label{ergodic-AB}
    Let $f:M\to M$ be a $C^{1+\alpha}$ conservative $AB$-system with a compact accessibility class of a closed manifold. If $f$ is transitive, then it is ergodic.
\end{thm}

\begin{proof}
    By definition, every $AB$-system has a one-dimensional center bundle. As shown in Theorem~\ref{c-dim-one}, either there exists an $su$-foliation by compact accessibility classes, or there are finitely many compact accessibility classes, all of which are periodic. In the first case, by Theorem~\ref{Ham-2.4}, $f$ is topologically conjugate to $A\times R_{\theta}$, where $A:N\to N$ is an Anosov automorphism of a compact nilmanifold and $R_{\theta}:\mathbb{S}^1\to \mathbb{S}^1$ is an irrational rotation with $\theta\in \mathbb{R}\setminus\mathbb{Q}$. This implies that $f$ is ergodic. We note that $C^{1+\alpha}$ regularity is sufficient to establish ergodicity in this case.

    In the second case, let $L_i$, $i=1,\dots,k$, denote all the compact accessibility classes and set $\mathcal{L}=\bigcup_{i=1}^{k}L_i$. The complement $\mathcal{D}:=M\setminus\mathcal{L}$ is a non-empty open set. By case (3) of Theorem~\ref{Ham-2.4}, there exists a continuous surjection $p:M\to \mathbb{S}^1$ that maps $\mathcal{D}$ onto an open set $U\subset\mathbb{S}^1$ and maps $\mathcal{L}$ to a finite set of points $z_i\in \mathbb{S}^1$, $i=1,\dots,k$. These points divide $\mathbb{S}^1$ into $k$ connected open intervals $I_i\subset U$, $i=1,\dots,k$. By Theorem~\ref{Ham-2.4}(3), each preimage $D_i:=p^{-1}(I_i)$ is a single open accessibility class of $f$, and we have $\mathcal{D}=\bigcup_{i=1}^kD_i$. Both $\mathcal{L}$ and $\mathcal{D}$ are $f$-invariant. By transitivity, after relabeling the sets $D_i$ and $I_i$ appropriately, we may assume that for any $i,j\in\{1,\dots,k\}$,
    \[
    f^i(D_j)=D_{i+j \pmod{k}}.
    \]
    Let $\mu$ be the $f$-invariant volume measure on $M$. Clearly $\mu(\mathcal{L})=0$. By $f$-invariance, for each $i\in\{1,\dots,k\}$,
    \[
    \mu(D_i)=\mu(f^i(D_1))=\mu(D_1)=\frac{1}{k},
    \]
    because $\mu\bigl(\bigcup_{i=1}^kD_i\bigr)=1$.

    Now, fix a continuous function $\phi\in C^0(M)$. For $\mu$-almost every $x\in M$, we have $x\in\bigcup_{i=1}^kD_i$. Write $x\in D_{i_0}$ for some $i_0$. Then, for each $m=0,\dots,k-1$, the points $f^{ik+m}(x)$ for $i=0,1,\dots$ lie in $D_{i_0+m \pmod{k}}$. By Theorem~\ref{acc->erg}, the restriction of $f^k$ to each $D_i$ is ergodic with respect to the conditional measure $\mu_{D_i}:=\frac{1}{\mu(D_i)}\mu\circ \mathds{1}_{D_i}=k\,\mu\circ\mathds{1}_{D_i}$. Therefore, for each $m$,
    \[
    \lim_{j\to\infty}\frac{1}{j}\sum_{i=0}^{j-1}\phi\bigl(f^{ik+m}(x)\bigr)=\int_{D_{i_0+m \pmod{k}}}\phi\,d\mu_{D_{i_0+m \pmod{k}}}.
    \]
    Consequently,
    \begin{align*}
        \lim_{n\to\infty}\frac{1}{n}\sum_{i=0}^{n-1}\phi\bigl(f^i(x)\bigr)
        &=\frac{1}{k}\sum_{m=0}^{k-1}\lim_{j\to\infty}\frac{1}{j}\sum_{i=0}^{j-1}\phi\bigl(f^{ik+m}(x)\bigr)\\
        &=\frac{1}{k}\sum_{m=0}^{k-1}\int_{D_{i_0+m \pmod{k}}}\phi\,d\mu_{D_{i_0+m \pmod{k}}}\\
        &=\frac{1}{k}\sum_{m=0}^{k-1}k\int_{D_{i_0+m \pmod{k}}}\phi\,d\mu\\
        &=\sum_{m=0}^{k-1}\int_{D_{i_0+m \pmod{k}}}\phi\,d\mu\\
        &=\int_{\bigcup_{i=1}^k D_i}\phi\,d\mu\\
        &=\int_M\phi\,d\mu.
    \end{align*}
    Since this holds for every $\phi\in C^0(M)$, the map $f$ is ergodic with respect to $\mu$. This completes the proof.
\end{proof}

We now consider the case of general partially hyperbolic diffeomorphisms in dimension three, and present the proof of Theorem~\ref{sol}.

\begin{proof}[Proof of Theorem~\ref{sol}]
    If $f$ is accessible, then it is ergodic by Theorem~\ref{acc->erg}. If $f$ is a DA diffeomorphism on $\mathbb{T}^3$, then it is also ergodic by \cite{GS-DA}. By Theorem~\ref{NW}, the only remaining case is when $f$ is an infra-$AB$-system with a compact accessibility class. We can freely pass to a finite cover and consider an iterate $\hat{f}^k$, since ergodicity of $\hat{f}^k$ implies ergodicity of $f$. Hence, without loss of generality, we assume that $f$ itself is an $AB$-system with a compact accessibility class. Then, Theorem~\ref{ergodic-AB} applies directly, yielding the ergodicity of $f$. This completes the proof.
\end{proof}


\section*{Acknowledgements}
I am grateful to Rafael Potrie for his valuable suggestions and to Meysam Nassiri for the conversations on his conjecture.

\bibliographystyle{alpha}
\bibliography{ref}
\appendix

\section{Chain transitivity}\label{subsec-chain-transitive}

We provide some equivalence statements of chain transitivity for homeomorphisms and flows.

Let $X$ be a connected complete metric space and $f:X\to X$ be a continuous map. For $f$ and two points $x,y\in X$, an \emph{$(\epsilon, f)$-chain} from $x$ to $y$ is a finite sequence $\{x_i\}_{i=0}^n$ with $x_0=x$ and $x_n=y$ such that $d(f(x_i),x_{i+1})\leq \epsilon$ for each $i=0,\dots,n-1$. When the dynamics are clear from context, we simply say an \emph{$\epsilon$-chain}.

Given a point $x\in X$, the \emph{chain transitive class} of $x$ is the set of all points $y\in X$ such that for every $\epsilon>0$, there exists an $\epsilon$-chain from $x$ to $y$. By continuity, one can verify that chain transitive classes are equivalence classes with respect to the chain relation. Moreover, each chain transitive class forms a closed subset of $X$. The map $f$ is called \emph{chain transitive} if for any $x,y\in X$ and any $\epsilon>0$, there exists an $\epsilon$-chain from $x$ to $y$; equivalently, $X$ itself is the unique chain transitive class. We say $f$ is \emph{chain recurrent} if for every $x\in X$ and every $\epsilon>0$, there exists a non-trivial $\epsilon$-chain from $x$ to itself.

\begin{lem}\label{equ-chain-map}
    Let $f:X\to X$ be a continuous map of a connected complete metric space. Then, the following statements are equivalent:
    \begin{enumerate}
        \item $f$ is chain recurrent;
        \item $f$ is chain transitive;
        \item any open set $U\subset X$ with $f(\overline{U})\subset U$ is either $\emptyset$ or $X$.
    \end{enumerate}
\end{lem}
\begin{proof}
    The equivalence of (2) and (3) can be found in \cite{CrovisierPotrie}. Here, we only present the equivalence of (1) and (2). Obviously the map $f$ is chain recurrent whenever it is chain transitive. Let us assume the chain recurrent property to show the chain transitivity. We first note that for each $x\in X$, the chain transitive class $\mathcal{R}(x)$ is a closed subset of $X$. It suffices to show that $\mathcal{R}(x)$ is open. 

    Let $y\in\mathcal{R}(x)$ and $z\in B_{\delta}(y)$ for a given small $\delta>0$. For any $\epsilon\geq 2\delta$, there is an $\epsilon$-chain $\{x_i\}_{i=1}^n$ from $x$ to $y$. As $f$ is chain recurrent, there is a $\delta$-chain $\{y_i\}_{i=0}^m$ from $y$ to itself. It follows that
    $$
    d(f(y_{m-1}),z)\leq d(f(y_{m-1}),y_m=y)+d(y,z)\leq 2\delta\leq \epsilon. $$
    Thus, the sequence $\{x=x_0, x_1, \dots,x_n, y=y_0,\dots, y_{m-1},z\}$ forms an $\epsilon$-chain from $x$ to $z$.

    Now we consider any small $\epsilon<2\delta$. By continuity, there exists $0<\epsilon'\leq \epsilon/4$ such that for any $u,v\in B_{\delta}(y)$ with $d(u,v)\leq 2\epsilon'$, we have $d(f(u), f(v))\leq \epsilon/2$. Pick a finite sequence of points $\{z_i\}_{i=0}^{k}\subset B_{\delta}(y)$ such that $y=z_0$ and $d(z_i,z_{i+1})\leq \epsilon'$, where we denote $z=z_{k+1}$. Since $f$ is chain recurrent, there are $\epsilon'$-chains $\{z_i^j\}_{j=0}^{n_i}$ with $z_i^0=z_i=z_i^{n_i}$ and $d(f(z_i^j),z_i^{j+1})\leq \epsilon'$ for $i=1,\dots, k$. By the choice of $\epsilon'$, we have that $d(f(y),f(z_1))\leq \epsilon/2$. It implies that
    $$
    d(f(y),z_1^1)\leq d(f(y),f(z_1))+d(f(z_1), z_1^1)\leq \epsilon/2+\epsilon'\leq \epsilon.
    $$
    Moreover, we have 
    $$
    d(f(z_1^{n_1-1}),z_2)\leq d(f(z_1^{n_1-1}),z_1^{n_1}=z_1)+d(z_1,z_2)\leq \epsilon'+\epsilon'=2\epsilon'.
    $$
    Then, we can deduce 
    $$
    d(f^2(z_1^{n_1-1}),z_2^1)\leq d(f^2(z_1^{n_1-1}),f(z_2))+d(f(z_2),z_2^1)\leq \epsilon/2+\epsilon'\leq \epsilon.
    $$
    One can proceed this argument inductively to conclude that the sequence 
    $$
    y,z_1^1, \dots, z_1^{n_1-1}, f(z_1^{n_1-1}), z_2^1, \dots, z_2^{n_2-1}, f(z_2^{n_2-1}), z_3^1, \dots, z_k^{n_k-1}, z
    $$
    is an $(\epsilon/2+\epsilon')$-chain. In particular, it is an $\epsilon$-chain. Therefore, the sequence
    $$
    x=x_0, x_1,\dots, x_n, y,z_1^1, \dots, z_1^{n_1-1}, f(z_1^{n_1-1}), z_2^1, \dots, z_2^{n_2-1}, f(z_2^{n_2-1}), z_3^1, \dots, z_k^{n_k-1}, z
    $$
    produces an $\epsilon$-chain from $x$ to $z$, which concludes $z\in\mathcal{R}(x)$. Hence, we finish the proof.
\end{proof}

We can define analogous notions of dynamical recurrence for flows. Let $X$ be a connected complete metric space and $\phi_t:X\to X$ be a continuous flow generated by a continuous non-singular vector field $\frac{\partial\phi_t}{\partial t}|_{t=0}$. For any two points $x,y\in X$, an \emph{$(\epsilon,\phi_t)$-chain} from $x$ to $y$ is defined as a finite sequence of tuples $\{(x_i,t_i)\}_{i=0}^n$ with $x_0=x$, $x_n=y$, and $t_i\geq 0$ such that $d(\phi_{t_i}(x_i),x_{i+1})\leq \epsilon$ for each $i=0,\dots, n-1$. When the flow $\phi_t$ is clear from context, we simply say an $\epsilon$-chain. Using the notion of $\epsilon$-chain for a flow, analogous to the case of maps, we can define the \emph{chain transitive class}, \emph{chain transitive flow}, and \emph{chain recurrent flow}.

We can state a similar result for flows as in Lemma~\ref{equ-chain-map}.
\begin{lem}\label{equ-chain-flow}
    Let $\phi_t: X\to X$ be a continuous flow on a connected complete metric space. Then, the following statements are equivalent:
    \begin{enumerate}
        \item $\phi_t$ is chain recurrent;
        \item $\phi_t$ is chain transitive;
        \item any open set $U\subset X$ such that $\phi_{t_0}(\overline{U})\subset U$ for some $t_0>0$ is either $\emptyset$ or $X$.
    \end{enumerate}
\end{lem}

The proof follows the same argument as in Lemma~\ref{equ-chain-map} with minor modifications. We only note that in the proof of Lemma~\ref{equ-chain-map}, we used the continuity of $f$ in one step to obtain the number $\epsilon'$. In the flow case, we fix $T>0$ and use the continuity of the flow on the interval $[0,T]$ to choose $\epsilon'$ such that for any $u,v\in B_{\delta}(y)$ with $d(u,v)\leq 2\epsilon'$, we have $\max\limits_{0\leq t\leq T}d(\phi_t(u),\phi_t(v))\leq \epsilon/2$. Given $T>0$, if $t_i>T$ for some pair $(x_i,t_i)$ in a chain, we can subdivide the flow segment $\phi_{[0,t_i]}(x_i)$ into finitely many pieces $\phi_{[0,t_i^1]}(x_i^1),\phi_{[0, t_i^2]}(x_i^2), \dots, \phi_{[0,t_i^k]}(x_i^k)$ such that $x_i^j=\phi_{\sum_{s=1}^j t_i^s}(x_i)$, $\sum_{s=1}^k t_i^s = t_i$, and $t_i^s \leq T$ for $s,j \in \{1,\dots, k\}$. We omit the proof of Lemma~\ref{equ-chain-flow}.

When a flow exhibits hyperbolicity, the characterization of chain transitivity can be strengthened as follows.

\begin{lem}\label{chain-transitive}
    For any (topological) Anosov flow $\phi_t$ on a closed manifold, the following statements are equivalent:
    \begin{enumerate}
        \item $\phi_t$ is chain transitive;
        \item its weak-stable and weak-unstable foliations are minimal;
        \item its non-wandering set is the whole manifold;
        \item $\phi_t$ is transitive.
    \end{enumerate}
\end{lem}
\begin{proof}
    Clearly, we have $(4)\Rightarrow (3)\Rightarrow (1)$. Thus, it suffices to prove the implications $(1)\Rightarrow (2)$ and $(2)\Rightarrow (4)$. Denote by $\mathcal{W}^{wu}$ (resp. $\mathcal{W}^{ws}$) the weak-unstable (resp. weak-stable) foliation of $\phi_t$, and for a point $x\in M$ denote by $W^{wu}(x)\in\mathcal{W}^{wu}$ (resp. $W^{ws}(x)\in\mathcal{W}^{ws}$) the corresponding leaf through $x$.

    First, assume (1) and suppose that $W^{wu}(x)$ is not dense for some $x\in M$. Then its closure $\overline{W^{wu}(x)}$ is a closed $\phi_t$-invariant proper subset that is saturated by $\mathcal{W}^{wu}$. For sufficiently small $\epsilon>0$, the set $Q:=W^{ws}_{\epsilon}(\overline{W^{wu}(x)})=\bigcup_{y\in\overline{W^{wu}(x)}}W^{ws}_{\epsilon}(y)$ is a proper subset of $M$. By transversality of $\mathcal{W}^{wu}$ and $\mathcal{W}^{ws}$, the set $Q$ is open. Moreover, there exists $t_0>0$ such that $\phi_{t_0}(\overline{Q})\subset Q$, which contradicts chain transitivity by Lemma~\ref{equ-chain-flow}. Similarly, if $W^{ws}(x)$ is not dense, then the set $P:=W^{wu}_{\epsilon}(\overline{W^{ws}(x)})$ is a proper open subset for sufficiently small $\epsilon$. Furthermore, there exists $t_0>0$ with $\overline{P}\subset\phi_{t_0}(P)$. This implies that the complement $L:=\overline{P}^c$ is a proper open set satisfying $\phi_{t_0}(\overline{L})\subset L$, again contradicting chain transitivity. Hence, (1) implies (2).

    Now assume (2); that is, the weak-unstable foliation is minimal. Given an arbitrarily small $\delta>0$, there exists a uniform constant $R_{\delta}>0$ such that every disk of radius $R_{\delta}$ in any leaf of $\mathcal{W}^{wu}$ intersects every $\delta$-ball in $M$. Let $U$ and $V$ be arbitrary open sets. Choose $\delta>0$ small enough so that $V$ contains a $\delta$-ball. Take a leaf $F\in\mathcal{W}^{wu}$ intersecting $U$ and let $D\subset U\cap F$ be a disk in $F$. Then there exists $T>0$ such that $\bigcup_{t\geq T}\phi_{t}(D)$ contains a disk of radius $R_{\delta}$ in $F$. By the choice of $R_{\delta}$, this union intersects $V$. Hence, $\phi_t$ is transitive, proving $(2)\Rightarrow (4)$.
\end{proof}

\section{Partial hyperbolicity and Anosov flows}

Here, we revisit another interesting equivalence relation between the existence of transitive partially hyperbolic diffeomorphisms and that of transitive Anosov flows in dimension three, which follows from the result of \cite{FP-classify} (see also Theorem~\ref{CAF}). See \cite[Question 1]{BGH-example}\cite[Question 5.7]{CHHU18}\cite[Question 9]{Potrie18}\cite{FP-gafa,Feng25} for the motivation and some background.

\begin{thm}\label{diffeo<=>flow}\cite{FP-classify}
    Let $M$ be a closed 3-manifold whose fundamental group has exponential growth. Then, $M$ admits a transitive partially hyperbolic diffeomorphism if and only if it supports a transitive Anosov flow.
\end{thm}

Indeed, this theorem can actually be strengthened under a weaker assumption (chain transitivity) while yielding a stronger conclusion (the topologically mixing property) as given in Theorem~\ref{diffeo->flow} and Proposition~\ref{flow->diffeo}.

\begin{thm}\label{diffeo->flow}\cite{FP-classify}
    Let $M$ be a closed 3-manifold whose fundamental group has exponential growth. If $M$ admits a chain-transitive partially hyperbolic diffeomorphism $f:M\to M$, then it admits a topologically mixing smooth Anosov flow.
\end{thm}

Given a transitive Anosov flow, its time-one map is clearly a partially hyperbolic diffeomorphism whose center foliation consists of flow lines. Furthermore, by discretizing a given Anosov flow, one can construct uncountably many distinct partially hyperbolic diffeomorphisms that are leaf conjugate to each other. However, these discretized Anosov flows are not necessarily transitive even if the original Anosov flow is transitive or topologically mixing.

The following proposition asserts that one can always find a transitive partially hyperbolic diffeomorphism (indeed, many such) that is a time-constant map of \emph{any given} chain transitive Anosov flow.

\begin{prop}\label{flow->diffeo}
    Let $\phi_t$ be a chain-transitive Anosov flow on a closed 3-manifold $M$. Then, there exists a constant $\alpha\in\mathbb{R}^+$ such that the time-$\alpha$ map $\phi_{\alpha}$ is transitive. Moreover, if $\phi_t$ is not a suspension flow over an Anosov diffeomorphism on $\mathbb{T}^2$, then $\phi_{\alpha}$ is topologically mixing.
\end{prop}
\begin{proof}
    By Lemma~\ref{chain-transitive}, the flow $\phi_t$ is transitive and has minimal weak-stable and weak-unstable foliations. By the Anosov closing lemma, transitivity of $\phi_t$ implies that periodic orbits are dense and there are countably many of them. Choose an irrational number $\alpha>0$ that is rationally independent of the periods of all periodic orbits. The time-$\alpha$ map $f:=\phi_{\alpha}$ is a partially hyperbolic diffeomorphism. Let $U$ and $V$ be any two open sets. Take a $\phi_t$-periodic point $p\in U$ and a small neighborhood $D\subset U$ contained in the weak-unstable leaf $W^{wu}_{\phi_t}(p)$. By the choice of $\alpha$, we have $\bigcup_{i\in \mathbb{N}}f^i(D)=W^{wu}_{\phi_t}(p)$. Since the weak-unstable foliation of $\phi_t$ is minimal, the leaf $W^{wu}_{\phi_t}(p)$ is dense, so $\bigcup_{i\in \mathbb{N}}f^i(D)$ intersects $V$. Hence, $f$ is transitive. If $\phi_t$ is not a suspension flow over an Anosov diffeomorphism on $\mathbb{T}^2$, then by \cite[Theorem 1.8]{Plante72}, the strong unstable (and strong stable) manifolds of $\phi_t$ (and hence of $f$) are dense in $M$. Consequently, the diffeomorphism $f$ is topologically mixing.
\end{proof}

As emphasized earlier, the previous result relies on a given transitive Anosov flow and a time-constant map of that flow. However, if one is only concerned with existence, we can perturb the time-$t$ map $\psi_t$ of any transitive Anosov flow to obtain a topologically mixing discretized Anosov flow $f$ associated with an Anosov flow that may differ from the original. Indeed, by the denseness of accessibility \cite{HHU08invent}, we can perform an arbitrarily $C^1$-small perturbation of the time-$t_0$ map to produce an accessible volume-preserving partially hyperbolic diffeomorphism $f$. Note that one only need this perturbation when $\pi_1(M)$ is not virtually solvable \cite{FP-adv}. This diffeomorphism $f$ is mixing (and in fact satisfies the K-property) with respect to the volume measure \cite{HHU08invent,BW10}, and hence it is topologically mixing. Moreover, $f$ is a discretized Anosov flow that is leaf conjugate to the given time-$t_0$ map. The plaque expansiveness of $\psi_{t_0}$ ensures the existence of a leaf conjugacy \cite[Theorem 7.1]{Hirsch-Pugh-Shub}, and consequently $f$ preserves every center leaf. Therefore, $f$ is a discretized Anosov flow \cite{Bonatti-Wilkinson05,Martinchich23} that is both accessible and topologically mixing.

The above discussion yields the following statement.

\begin{prop}
    Let $M$ be a closed 3-manifold supporting a chain transitive Anosov flow. Then, there exists a discretized Anosov flow that is accessible and topologically mixing.
\end{prop}

The above result does not imply that, given the existence of a transitive Anosov flow, every discretization of this flow is transitive or has the entire manifold as its non-wandering set. Indeed, one can easily construct partially hyperbolic diffeomorphisms with arbitrarily many disjoint non-wandering sets from a transitive suspension Anosov flow. Even when the ambient 3-manifold has non-virtually solvable fundamental group, there exist partially hyperbolic diffeomorphisms with proper attractors that are discretized from a transitive Anosov flow \cite{BonattiGuelman10}.

Note that this proposition applies to any chain transitive Anosov flow, including suspensions. However, the discretized Anosov flow obtained is not necessarily a time-constant map, in contrast to the construction in Proposition~\ref{flow->diffeo}. Moreover, even if it is a time-constant map, the associated (topological) Anosov flow may differ from the original.

\end{document}